\documentclass[12pt, a4paper]{amsart}
\usepackage{amsmath}
\usepackage{geometry,amsthm,graphics,tabularx,amssymb,shapepar}
\usepackage{amscd}
\usepackage[all]{xypic}
\allowdisplaybreaks

\newcommand{\BG}{{\mathbb {G}}}
\newcommand{\BH}{{\mathbb {H}}}

\newcommand{\CF}{{\mathcal {F}}}

\newcommand{\CL}{{\mathcal {L}}}

\newcommand{\CN}{{\mathcal {N}}}
\newcommand{\CO}{{\mathcal {O}}}

\newcommand{\CS}{{\mathcal {S}}}

\newcommand{\Ad}{{\mathrm{Ad}}}

\newcommand{\End}{{\mathrm{End}}}

\newcommand{\GL}{{\mathrm{GL}}}

\newcommand{\Hom}{{\mathrm{Hom}}}

\newcommand{\rank}{{\mathrm{rank}}}

\newcommand{\rk}{{\mathrm{k}}}

\newcommand{\SL}{{\mathrm{SL}}}

\newcommand{\tr}{{\mathrm{tr}}}

\newcommand{\con}{\textit{C}}

\newcommand{\od}{\operatorname{d}}

\newcommand{\oH}{\operatorname{H}}

\newcommand{\oS}{\operatorname{S}}

\newcommand{\oM}{\operatorname{M}}
\newcommand{\oG}{\operatorname{G}}
\newcommand{\oZ}{\operatorname{Z}}
\newcommand{\oV}{\operatorname{V}}
\newcommand{\oD}{\textit{D}}

\newcommand{\oE}{\operatorname{E}}

\newcommand{\g}{\mathfrak g}

\newcommand{\h}{\mathfrak h}

\renewcommand{\v}{\mathfrak v}

\renewcommand{\l}{\mathfrak l}

\newcommand{\s}{\mathfrak s}

\renewcommand{\sl}{\mathfrak s \mathfrak l}

\renewcommand{\rk}{\mathrm k}

\newcommand{\Z}{\mathbb{Z}}
\newcommand{\C}{\mathbb{C}}
\newcommand{\R}{\mathbb R}

\newcommand{\abs}[1]{\lvert#1\rvert}

\newcommand{\la}{\langle}
\newcommand{\ra}{\rangle}

\newcommand{\be}{\begin {equation}}
\newcommand{\ee}{\end {equation}}
\newcommand{\bee}{\begin {equation*}}
\newcommand{\eee}{\end {equation*}}

\newcommand{\cf}{\emph{cf.}~}

\theoremstyle{Theorem}

\theoremstyle{Theorem}
\newtheorem{introconjecture}{Conjecture}

\newtheorem{introtheorem}[introconjecture]{Theorem}

\theoremstyle{Theorem}
\newtheorem{lem}{Lemma}[section]

\theoremstyle{Theorem}

\theoremstyle{Plain}

\theoremstyle{Definition}
\newtheorem{dfn}{Definition}[section]

\newtheorem{lemd}[dfn]{Lemma}
\newtheorem{prpd}[dfn]{Proposition}

\newtheorem{thmd}[dfn]{Theorem}

\begin{document}

\title[Linear periods and Shalika periods]{Uniqueness of twisted linear periods and twisted Shalika periods}

\author{Fulin Chen}

\address{Department of Mathematics, Xiamen University,
 Xiamen, 361005, China} \email{chenf@xmu.edu.cn}

\author{Binyong Sun}

\address{Academy of Mathematics and Systems Science, Chinese Academy of
Sciences \& University of Chinese Academy of Sciences,  Beijing, 100190, China} \email{sun@math.ac.cn}



\begin{abstract}
Let $\rk$ be a local field of characteristic zero. Let $\pi$ be an
irreducible admissible smooth representation of $\GL_{2n}(\rk)$. We prove that for all but countably many characters $\chi$ of $\GL_n(\rk)\times
\GL_n(\rk)$, the space of $\chi$-equivariant (continuous in the archimedean case) linear
functionals on $\pi$ is at most one dimensional.
Using this, we prove the uniqueness of twisted Shalika models.
\end{abstract}

 \maketitle


\section{Introduction}

Let $\rk$ be a local field of characteristic zero.  The Shalika subgroup of the general linear group $\GL_{2n}(\rk)$ ($n\geq 0$) is defined to be
\begin{equation}
  \oS_n(\rk):=\left\{           \left[
            \begin{array}{cc}
              a&b\\ 0&a
              \end{array}
           \right]\mid a\in \GL_n(\rk), \, b\in \oM_n(\rk)
           \right\},
\end{equation}
where ``$\oM_n$" indicates the algebra of $n\times n$-matrices. Fix a character $\psi_{\oS_n}$ on $\oS_n(\rk)$ such that
\begin{equation}\label{psik}
 \psi_{\oS_n}\left( \left[ \begin{array}{cc}
              1&b\\ 0&1
              \end{array}
              \right]\right)=\psi_\rk(\tr(b)), \quad \textrm{for all $b\in \oM_n(\rk)$,}
\end{equation}
where $\psi_\rk: \rk\rightarrow \C^\times$ is a non-trivial unitary character. We will prove the following uniqueness result in this paper.

\begin{introtheorem}\label{unique0}
For every irreducible admissible smooth representation $\pi$ of
$\GL_{2n}(\rk)$, the space
\be\label{sshalika}
  \Hom_{\oS_n(\rk)}(\pi, \psi_{\oS_n})
  \ee
   is at most one dimensional.
\end{introtheorem}
Here and henceforth, when $\rk$ is archimedean, by an admissible smooth representation of $\GL_m(\rk)$ ($m\geq 0$)
we mean a Casselman-Wallach representation of it.  Recall that a
representation of a real reductive group is called a
Casselman-Wallach representation if it is Fr\'{e}chet, smooth, of
moderate growth, and its Harish-Chandra module has finite length.
The reader may consult \cite{C}, \cite[Chapter 11]{W} or \cite{BK}
for details about Casselman-Wallach representations. In the non-archimedean case, the notion of ``admissible smooth representation" retains the usual meaning.

Set
\begin{equation}\label{dn}
  \oD_n(\rk):=\left \{\left[ \begin{array}{cc}
              a&0\\ 0&a
              \end{array}
              \right]\mid a\in \GL_n(\rk)\right\}\subset \oS_n(\rk).
              \end{equation}
When $\psi_{\oS_n}$ has trivial restriction to $\oD_n(\rk)$, Theorem \ref{unique0} is proved in \cite{JR}
for the non-archimedean case and \cite{AGJ} for the archimedean case. This implies the same result when the restriction of $\psi_{\oS_n}$  to $\oD_n(\rk)$ is the square of a character.

A non-zero element of the space \eqref{sshalika} is called a local Shalika period of $\pi$.
Using the Langlands lift to $\GL_{2n}$, local  Shalika periods and their global analogues are fundamental to the study of standard L-functions of $\mathrm{GSpin}_{2n+1}$. See \cite[Section 3]{GR} or \cite{AsG} for example.
Similar to the untwisted case \cite{JR,AGJ}, the proof of Theorem \ref{unique0} is based on Shalika zeta integrals \cite{FJ} and the following uniqueness result.

\begin{introtheorem}\label{uniquel}
Let $\pi$ be an irreducible admissible smooth representation of
$\GL_{2n}(\rk)$. Then for all but countably many  (finitely many in the non-archimedean case)
 characters $\chi$ of $\GL_n(\rk)\times \GL_n(\rk)$, the space
\be\label{unihom}
  \Hom_{\GL_n(\rk)\times \GL_n(\rk)}(\pi, \chi)
  \ee
   is at most one dimensional.
\end{introtheorem}

A non-zero element of the space \eqref{unihom} is called a local linear period of $\pi$.
When $\chi$ is the trivial character,  the uniqueness of local linear periods is
proved by Jacquet and Rallis \cite[Theorem 1.1]{JR} for the non-archimedean case,
and by Aizenbud and Gourevitch \cite[Theorem 8.2.4]{AG} for the
archimedean case.

The reader is referred to \cite{FJ,JR} for the role of local linear periods and their global analogues in the  study of L-functions.
In a recent work of the second named author, Theorem \ref{uniquel} is used in the proof  of a non-vanishing
assumption which is critical to the arithmetic study of special values of  L-functions for
$\mathrm{GSpin}_{2n+1}$. See \cite[Section 4]{Sun} for details. This is the original motivation of this paper.

Let us now introduce a technical notion on  characters of $\GL_n(\rk)\times \GL_n(\rk)$.
We use $|\cdot|$ to denote the normalized absolute value on $\rk$, and we also use it to stand for the character $t\mapsto |t|$ of $\rk^\times$.
We say that a character of $\rk^\times$ is pseudo-algebraic if it has the form
\[
  t\mapsto \left\{ \begin{array}{ll}
                          1, \quad &\textrm{if $\rk$ is non-archimedean};\\
                          t^m, \quad & \textrm{if $\rk=\R$};\\
                          \iota(t)^{m} \cdot \iota'(t)^{m'},\quad  & \textrm{if $\rk\cong \C$},
                           \end{array}
                           \right.
\]
where $m, m'$ are non-negative integers, and $\iota, \iota'$ are the two distinct topological isomorphisms from $\rk$ to $\C$.

A character $\gamma$ of $\GL_n(\rk)$ is said to be good if it equals $\eta\circ \det$ for some  character $\eta$ of $\rk^\times$ such that
\[
  \eta^{2r} \cdot | \cdot |^{-m} \ \textrm{ is not  pseudo-algebraic}
\]
  for all $r\in \{\pm 1, \pm 2, \cdots, \pm n\}$ and all $m\in\{1,2, \cdots, 2n^2\}$. Note that $\gamma$ is good if and only if so is $\gamma^{-1}$, and all but countably many (finitely many in the non-archimedean case)
characters of $\GL_n(\rk)$ are good. A character  $\chi=\gamma_0\otimes \gamma_1$ of $\GL_n(\rk)\times \GL_n(\rk)$ is said to be good if
the character $\gamma_0\gamma_1^{-1}$ of $\GL_n(\rk)$ is good.

\begin{introtheorem}\label{uniquef}
Let $f$ be a generalized function on $\GL_{2n}(\rk)$ and let $\chi$ be a good  character of $\GL_n(\rk)\times \GL_n(\rk)$.
If for every
$h\in \GL_n(\rk)\times \GL_n(\rk)$,
\begin{equation}\label{inf0}
   f(hx)=f(xh)=\chi(h)f(x),\quad x\in \GL_{2n}(\rk),
\end{equation}
as generalized functions on $\GL_{2n}(\rk)$, then
\[
   f(x)=f(x^t).
\]
\end{introtheorem}

Here and as usual, a superscript ``$t$" indicates the transpose of a
matrix. For the usual notion of generalized functions, see
\cite[Section 2.1]{JSZ} (archimedean case), and \cite[Section
2]{Sun1} (non-archimedean case), for examples.

Let $\pi$ be  an irreducible admissible smooth representation  of $\GL_{2n}(\rk)$, and let $\chi$ be a character of $\GL_n(\rk)\times \GL_n(\rk)$. By taking the generalized matrix coefficient as in \cite{SZ}, we produce a nonzero generalized function satisfying \eqref{inf0} from every nonzero vector in
 \[
   \Hom_{\GL_n(\rk)\times \GL_n(\rk)}(\pi, \chi)\otimes \Hom_{\GL_n(\rk)\times \GL_n(\rk)}(\pi^\vee, \chi^{-1}).
 \]
 Here and as usual, a superscript ``$\,^\vee$" indicates the contragredient representation. It is well known that (\cf \cite{GK})
 \[
   \Hom_{\GL_n(\rk)\times \GL_n(\rk)}(\pi, \chi)\cong\Hom_{\GL_n(\rk)\times \GL_n(\rk)}(\pi^\vee, \chi^{-1}).
 \]
 Thus by  Gelfand-Kazhdan criterion (\cf \cite[Theorem 2.3]{SZ}), Theorem \ref{uniquef} implies that
\be\label{atmone}
 \textrm{the  space  \eqref{unihom} is at most one dimensional if $\chi$ is good.}
 \ee
Furthermore, it is clear that the space  \eqref{unihom} is non-zero only if the restriction of $\chi$ to the center of $\GL_{2n}(\rk)$
coincides with the central character of $\pi$.
Therefore Theorem \ref{uniquel} follows from \eqref{atmone}.

Observe that the trivial character of $\GL_n(\rk)\times \GL_n(\rk)$ is good. Thus  in particular we have proved  the uniqueness of untwisted linear periods, which is first proved in \cite{JR,AG}.
Note that Theorem \ref{uniquef} is not previously known even
when $\chi$ is trivial. What Jacquet-Rallis and Aizenbud-Gourevitch
have proved is that if \eqref{inf0} holds for trivial $\chi$, then
$f(x)=f(x^{-1})$. However, this does not hold for general characters.
More precisely, suppose that a nonzero generalized function  $f$ satisfies \eqref{inf0}.
  If $f$ is invariant under the inverse map, then
\[\chi(h)f(x)=f(xh)=f(h^{-1}x^{-1})=\chi(h^{-1})f(x^{-1})=\chi^{-1}(h)f(x).\]
This forces  $\chi$ to be a quadratic character.  Hence the method of \cite{JR,AG} can not be applied directly to the general case.

By linearization,  Theorem \ref{uniquef} is reduced to the following three assertions.

\begin{introtheorem}\label{tlin}
(a) Let $f$ be a generalized function on $\oM_{n}(\rk)$ such that for all $g\in \GL_n(\rk)$,
\[
  f(gxg^{-1})=f(x), \quad x\in \oM_n(\rk).
\]
 Then $f(x)=f(x^t)$. $\smallskip$

\noindent
(b)  Let $f$ be a generalized function on $\oM_{n}(\rk)\times \oM_{n}(\rk) $ such that for all $g,h\in \GL_n(\rk)$,
\[
  f(gxh^{-1}, hyg^{-1})=f(x,y), \quad (x,y)\in \oM_n(\rk)\times \oM_{n}(\rk).
\]
 Then $f(x,y)=f(x^t, y^t)$. $\smallskip$

\noindent
(c) Let $\gamma$ be a good  character of $\GL_n(\rk)$ and let $f$ be a generalized function on $\oM_{n}(\rk)\times \oM_{n}(\rk) $ such that for all $g,h\in \GL_n(\rk)$,
\[
  f(gxh^{-1}, hyg^{-1})=\gamma(g)\gamma(h^{-1}) f(x,y), \quad (x,y)\in \oM_n(\rk)\times \oM_{n}(\rk).
\]
Then $f(x,y)=f(y^t, x^t)$.
\end{introtheorem}

Part (a) of Theorem \ref{tlin} is well-known,  \cf  \cite[Theorem 2.1]{SZ2},  \cite[Proposition 4.I.2]{MVW} and \cite{AGRS}. By the method of \cite{LST},  Part (b) of Theorem \ref{tlin} implies the following particular case of the multiplicity one result of local theta correspondence:
\begin{equation}\label{homgg}
 \dim \Hom_{\GL_n(\rk)\times \GL_n(\rk)} (\CS(\oM_n(\rk)), \pi\widehat \otimes \pi')\leq 1.
\end{equation}
Here, $\pi,\pi'$ are irreducible admissible smooth representations of $\GL_n(\rk)$; ``$\widehat \otimes$" stands for the completed projective tensor product in the archimedean case and the algebraic tensor product in the non-archimedean case; and
$\CS(\oM_n(\rk))$ is the space of Schwartz functions on $\oM_n(\rk)$ carrying the representation of  $\GL_n(\rk)\times \GL_n(\rk)$ by the left and right translations.
It is well known that the equality in  \eqref{homgg}  holds if and only if $\pi'\cong \pi^\vee$. This is a fundamental fact in the theory of Godement-Jacquet L-functions.

Part (c) of Theorem \ref{tlin} fails for some non-good characters. For example, set
\[
  f=\frac{\textrm{a Haar measure on $\oM_n(\rk)\times\{0\}$}}{\textrm{a Haar measure on $\oM_n(\rk)\times\oM_n(\rk)$}},
\]
which is a generalized function on $\oM_n(\rk)\times\oM_n(\rk)$ satisfying
\[
  f(gxh^{-1}, hyg^{-1})=\abs{\det(g)}^n \cdot \abs{\det(h)}^{-n}\cdot f(x,y), \quad (x,y)\in \oM_n(\rk)\times \oM_{n}(\rk),
\]
for all $g,h\in \GL_n(\rk)$. But the generalized functions $f(x,y)$ and $f(y^t, x^t)$ are not equal to each other unless $n=0$.  By this example, \cite[Remark 3.1.2]{AG} implies  that Theorem \ref{uniquef} fails for some non-good characters. But we do not know whether or not Theorem \ref{uniquel} fails for some non-good characters.

Here are a few words on the organization of the paper. In Section 2, we introduce the notions of graded involutive algebras and graded Hermitian modules, and consider Harish-Chandra
descents  and MVW-extensions on them. We also introduce some characters which will occur in the proof of Theorem \ref{uniquef}.
Theorem \ref{tlin} is proved in Section 3,  and a slight generalization of Theorem \ref{uniquef} (see Theorem  \ref{vanisht}) is proved in Section 4.  Finally, it is proved in Section 5 that  Theorem \ref{uniquel} implies Theorem \ref{unique0}.

The authors thank Avraham Aizenbud for helpful email communications, and thank  Dmitry Gourevitch for a critical bibliographical remark.
F. Chen was supported in part by the National Natural Science Foundation of China (No.11501478)  and the Fundamental Research Funds for the Central University (No.20720150003).
B. Sun was supported in part by the National Natural Science Foundation of China (No. 11525105, 11688101,
11621061
and 11531008).

\section{Graded Hermitian modules}

As in the Introduction, fix a local field $\rk$ of characteristic zero.

\subsection{Hermitian modules and MVW-extensions}
By an involutive algebra, we mean   a commutative semisimple finite-dimensional $\rk$-algebra equipped
with an involutive $\rk$-algebra automorphism of it.  We use $\tau$ to indicate the given involutive automoprhisms of various involutive algebras. Let $A$ be an involutive algebra in this subsection.  We say that $A$ is simple if it is non-zero, and has no non-zero proper $\tau$-stable ideal. This is equivalent to saying that $A$ is either a field or the product of two fields which are exchanged by $\tau$.
In general, $A$ is uniquely a product of simple involutive algebras.

Let $E$ be a Hermitian $A$-module, namely,
a  finitely generated $A$-module equipped with a non-degenerate $\rk$-bilinear map
$\la\,,\,\ra_E: E\times E\rightarrow A$ which satisfies that
\[
     \la u,v\ra_E=\la v,u\ra_E^\tau \quad\textrm{and}\quad \la a.u,v\ra_E=a\la u,
     v\ra_E,\quad a\in A,\, u,v\in E.
\]
Note that if $A$ is simple, then $E$ is free as an $A$-module.

Write $\oG(E)$ for the group of all $A$-module automorphisms of $E$ which preserve the Hermitian form.
The MVW-extension of $\oG(E)$, denoted by $\breve\oG(E)$, is defined to be the subgroup  of $\GL(E_\rk)\times \{\pm 1\}$ consisting of all pairs $(g,\delta)$ such that
either $\delta=1$ and  $g\in \oG(E)$, or
\[
  \delta=-1 \quad \textrm{and}  \quad\la g.u, g.v\ra_E=\la v, u\ra_E, \ u,v\in E.
\]
Here $E_\rk$ stands for the underlying $\rk$-vector space of $E$.
It is well-known that
the group $\breve \oG(E)$ contains $\oG(E)$ as a subgroup of index $2$ (\cite{MVW}).

We are particularly interested in the case when $A=\rk\times \rk$ and $\tau$ equals the coordinate exchange map.  In this case
\begin{equation}\label{gln}
  \oG(E)=\GL( e_1 E)\cong \GL_n(\rk), \quad (n:=\rank_A E),
\end{equation}
and
\[
  \breve \oG(E)\cong \{\pm 1\}\ltimes \GL_n(\rk),
\]
where $e_1$ denotes the element $(1,0)$ of $A$, and the semi-direct product is defined by the action
\[
(-1).g=g^{-t},\quad g\in \GL_n(\rk).
\]
\subsection{Graded modules}

By a graded algebra, we mean a commutative semisimple finite-dimensional $\rk$-algebra $A$, equipped
with a $\Z/2\Z$-grading $A=A_0\oplus A_1$ such that
\[
 1\in A_0,\quad\rk. A_i\subset A_i\quad \textrm{and}\quad A_i\cdot A_j\subset A_{i+j},\quad i,j\in \Z/2\Z.
\]
Let $A=A_0\oplus A_1$ be a graded algebra in this subsection.

\begin{dfn}
We say that $A$
is  complex if $A_1$ contains an invertible element of $A$. We say  that $A$ is real if $A_1=0$.
\end{dfn}

The following lemma is obvious.
\begin{lemd}\label{cac}
Let $A\rightarrow A'$ be a homomorphism of graded  algebras (that is, a $\rk$-algebra homomorphism preserving the gradings). If $A$ is complex, then $A'$ is also complex.
\end{lemd}

\begin{dfn}
A graded $A$-module is a finitely generated $A$-module $E$, equipped with a $\Z/2\Z$-grading $E=E_0\oplus E_1$ such that
\[
   A_i. E_j\subset E_{i+j} ,  \quad i,j\in \Z/2\Z.
\]
\end{dfn}

Let $E=E_0\oplus E_1$ be a graded $A$-module in this subsection.
\begin{dfn}
 We say that  $E$ is complex if  $E_0$ and $E_1$ are isomorphic to each other as $A_0$-modules.
 \end{dfn}

 The following lemma  is obvious.
 \begin{lemd}\label{compcri} Let $E=E'\oplus E''$ be a direct sum of graded $A$-modules.
 If two of $E, E', E''$ are complex, then so is the third one.
 \end{lemd}

Note that  $A\otimes_{A_0} E_0$ is naturally a graded $A$-module, and the  obvious $A$-module homomorphism
\be\label{e0a}
  A\otimes_{A_0} E_0\rightarrow E
\ee
is a homomorphism of graded $A$-modules, that is, it preserves the gradings.

\begin{lemd}\label{a0e}
If $A$ is complex, then $E$ is complex and the homomorphism \eqref{e0a} is an isomorphism.
\end{lemd}
\begin{proof}
 Take an invertible element $a\in A_1$. Then $A_1$ is a free $A_0$-module with a free generator $a$, and the multiplication by $a$ gives an $A_0$-module isomorphism $E_0\rightarrow E_1$. Thus the lemma follows.
\end{proof}

\subsection{Graded Hermitian modules and MVW-extensions}
\begin{dfn}A graded involutive algebra is a graded algebra $A=A_0\oplus A_1$ with an involutive automorphism $\tau$ on it which preserves the grading.
\end{dfn}

Thus every graded involutive algebra is a graded algebra as well as an involutive algebra.
From now on, let $A=A_0\oplus A_1$ be a graded involutive algebra.
Similar to before, we say that $A$ is simple if it is non-zero, and has no non-zero proper graded $\tau$-stable ideal. In general, $A$ is uniquely a product of   simple  graded involutive algebras.

We say that a graded involutive algebra is real or complex if it is so as a graded algebra.

\begin{lemd}\label{inva1}If $A$ is simple, then it is either real or complex.
\end{lemd}
\begin{proof}
If $A$ is not real, then there is a non-zero element $a$ in $A_1$ such that $a^\tau=\pm a$.
Note that $Aa$ is a non-zero graded $\tau$-stable ideal of $A$. Then $A=Aa$, which implies that $a$ is invertible.
\end{proof}

Note that $A_0$ is obviously an involutive algebra.
\begin{lemd}If $A$ is simple, then the involutive algebra $A_0$ is simple.
\end{lemd}
\begin{proof} If $A$ is real, then $A_0$ is obviously simple. So we  assume that $A$ is complex.
As in the proof of Lemma \ref{inva1}, take an invertible element $a\in A_1$ such that $a^\tau=\pm a$. Then $A_1=A_0 a$. Let $I_0$ be a non-zero involutive ideal of $A_0$.
Then $I_0\oplus I_0 a$ is a non-zero graded involutive ideal of $A$. Therefore $I_0\oplus I_0 a=A$ and $I_0=A_0$.
\end{proof}

\begin{dfn}
 A graded Hermitian $A$-module is a Hermitian $A$-module $E$, equipped with a $\Z/2\Z$-grading $E=E_0\oplus E_1$ such that
\[
   A_i. E_j\subset E_{i+j}  \quad \textrm{and}\quad \la E_i, E_j\ra_E\subset A_{i+j}, \quad i,j\in \Z/2\Z.
\]
\end{dfn}

Thus every graded Hermitian $A$-module is a Hermitian $A$-module as well as a graded $A$-module.
 From now on, let $E=E_0\oplus E_1$ be a graded Hermitian $A$-module. Note that both $E_0$ and $E_1$ are Hermitian $A_0$-modules:
 their Hermitian forms are given by taking the restrictions of $\la\,,\,\ra_E$. For every graded involutive quotient $A'$ of $A$  (a graded involutive quotient is a quotient by a $\tau$-stable graded ideal), the tensor product $A'\otimes_A E$ is obviously a graded Hermitian $A'$-module.

As before, denote by $E_\rk$ the underlying $\rk$-vector space of $E$. The endomorphism algebra $\End(E_\rk)$ is a $\Z/2\Z$-graded $\rk$-algebra:
\be\label{gradee}
  \End(E_\rk)={\End(E_\rk)}_0\oplus {\End(E_\rk)}_1,
\ee
where
\[
  {\End(E_\rk)}_i:=\{x\in \End(E_\rk)\mid x.E_j\subset E_{i+j}, \,j\in \Z/2\Z \},\quad i\in \Z/2\Z.
\]
For any $\Z/2\Z$-graded  vector space over $\rk$, we use ``$\,\bar{\ }\,$" to denote the involutive automorphism of it whose restriction to
the degree $i$ part is the multiplication by $(-1)^i$ ($i\in \Z/2\Z$). Specifically, this notation applies to $\End(E_\rk)$ and all graded involutive algebras.

 Denote by $\oH(E)$ the group of all $A$-module automorphisms of $E$ which preserve both the grading and the form $\la\,,\,\ra_E$. Note that
\[
  \oH(E)=\{g\in  \oG(E)\mid \bar g=g\}.
\]
Put
\[
  \oV(A):=\{a\in A^\times\mid a a^\tau=1=a\bar a\}.
\]
For each $\alpha\in \oV(A)$, write
 \[
   \breve\oH_\alpha(E):=\{(g,\delta)\in \breve \oG(E)\mid \bar g=g\textrm{ if }\delta=1;\ \bar g=\alpha g\textrm{ if }\delta=-1\}.
 \]
Note that $\breve\oH_\alpha(E)$ is a subgroup of $\breve \oG(E)$, and contains $\oH(E)$ as a subgroup of index $1$ or $2$. We call $\breve\oH_\alpha(E)$ the MVW-extension of $\oH(E)$ associated to $\alpha$.

\subsection{Harish-Chandra descent}\label{descent}
Associated to the group $\oG(E)$ we have the Lie algebra
\[
  \g(E):=\{x\in \End_A(E)\mid \la x.u, v\ra_E+\la u, x.v\ra_E=0,\,u,v\in E\}.
\]
It admits a natural $\Z/2\Z$-grading
\[
  \g(E)=\h(E)\oplus \v(E)
\]
where
\[
  \h(E):=\{x\in \g(E)\mid \bar x=x\}
  \]
   is the Lie algebra of $\oH(E)$, and
   \[
     \v(E):=\{x\in \g(E)\mid \bar x+x=0\}.
   \]
 Put
\[
  \oV(E):=\{x\in \oG(E)\mid x\bar x=1\}.
\]

Fix an element $s$ of $\oV(E)$ or $\v(E)$ which is semisimple in the sense that it is
semisimple as a $\rk$-linear operator on $E$. Denote by $A_s$ the finite-dimensional $\rk$-subalgebra of $\End_A(E)$ generated by
$s$ and the scalar multiplications from $A$. It is commutative and semisimple. Moreover, it is a graded involutive algebra:
the grading is induced by the grading \eqref{gradee}, and the involutive automorphism is induced by the anti-automorphism
\be\label{tau}
   \End_A(E)\rightarrow \End_A(E),\quad    x\mapsto   x^{\tau_E}
      \ee
specified by
\[
  \la x.u, v\ra_E=\la u, x^{\tau_E}. v\ra_E,\quad u,v\in E.
\]
 We call the graded involutive algebra $A_s$ a Harish-Chandra descent of $A$, and write $A_s=(A_s)_0\oplus (A_s)_1$ for the grading.

The natural $\rk$-algebra homomorphism $A\rightarrow A_s$ is clearly a homomorphism of graded involutive algebras, namely it preserves both the gradings and the involutions.
Assume that $E$ is faithful as an $A$-module throughout the rest of the paper. Then  the homomorphism
$A\rightarrow A_s$ is an embedding.
\begin{lemd}\label{as2}
Assume that  $A$ is simple and $s\in \oV(E)$. Then $A_s$ is complex, or the product of $A$  with a complex graded involutive algebra,
or the product of $A\times A$ with a complex graded involutive algebra. In the last case, the image of $s$ via the projection
$A_s\rightarrow A\times A$ is either $(1,-1)$ or $(-1,1)$.
\end{lemd}
\begin{proof}
 We have an $s$-stable graded Hermitian $A$-module decomposition
$E=E'\oplus E''$   such that
\[
  s': E'\rightarrow E', \quad u\mapsto s(u)
  \]
  has no eigenvalue $1$ or $-1$, and
  \[
  s'': E''\rightarrow E'',\quad u\mapsto s(u)
  \]  has no eigenvalue other than $\pm 1$.
Note that $s'\in\oV(E')$ and $s''\in\oV(E'')$. Form the Harish-Chandra descents $A_{s'}\subset \End_A(E')$ and  $A_{s''}\subset \End_A(E'')$.

We claim that the natural map
\be
 f: A_s\rightarrow A_{s'}\times A_{s''},\quad x\mapsto (x|_{E'}, x|_{E''})
\ee
is an isomorphism of graded involutive algebras. Indeed, it is easy to see that $f$ is an injective homomorphism of
graded involutive algebras.
Note that $s'-s'^{-1}$ is invertible as $\rk$-linear map on $E'$. Thus there exist $b_1, b_2,\cdots, b_r\in \rk^\times$ ($r\geq 1$) such that
\[1+b_1(s'-s'^{-1})+b_2(s'-s'^{-1})^2+\cdots+b_r(s'-s'^{-1})^r=0.\]
Together with the fact that $s''-s''^{-1}=0$, this implies
\[f(1+b_1(s-s^{-1})+\cdots +b_r(s-s^{-1})^r)=(0,1).\]
Thus $(0,1)$ is in the image of $f$. This easily implies that $f$ is surjective.

Finally,   $A_{s'}$ is complex since it contains the invertible element $s'-s'^{-1}\in (A_{s'})_1$.
Furthermore, $A_{s''}\cong 0$, $A$, or $A\times A$, if the set
\[
\{\epsilon=\pm 1\mid \epsilon \textrm{ is an eigenvalue of } s''\}
\]
 has cardinality $0$, $1$, or $2$, respectively. This proves the lemma.
\end{proof}

Similarly,  one has the following result for $s\in \v(E)$.
\begin{lemd}\label{as}
Assume that  $A$ is simple and $s\in \v(E)$. Then  $A_s$ is complex,  or the product of $A$ with a complex graded involutive algebra.
\end{lemd}
\begin{proof}
We have an $s$-stable graded Hermitian $A$-module decomposition
$E=E'\oplus E''$   such that
\[
  s': E'\rightarrow E', \quad u\mapsto s(u)
  \]
  has no eigenvalue $0$, and
  \[
  s'': E''\rightarrow E'',\quad u\mapsto s(u)
  \]
    has no eigenvalue other than $0$.
  As in the proof of Lemma \ref{as2}, the lemma follows by showing that $A_s\cong A_{s'}\times A_{s''}$, $A_{s'}$ is complex, and $A_{s''}$ is either zero or isomorphic to $A$.
  \end{proof}

Write $E_s$ for the space $E$ viewing as an $A_s$-module. Put $(E_s)_i:=E_i$ ($i\in \Z/2\Z$). Then $E_s=(E_s)_0\oplus (E_s)_1$ is a graded $A_s$-module. As in \cite[Lemma 3.1]{Sun1}, define a Hermitian form
\[
  \la\,,\,\ra_{E_s}: E_s\times E_s\rightarrow A_s
\]
on $E_s$ by requiring that
\[
  \tr_{A_s/\rk}(a \la u,v\ra_{E_s})
= \tr_{A/\rk}(\la a.u,v\ra_E), \quad u,v\in E, \, a\in A_s.
\]

\begin{lemd}\label{hgd}
One has that
\[
   \la (E_s)_i, (E_s)_j\ra_{E_s}\subset (A_s)_{i+j},\quad i,j\in \Z/2\Z.
\]
\end{lemd}
\begin{proof}
Let $u\in (E_s)_i$ and $v\in (E_s)_j$. For each $a\in A_s$, one has that
\begin{eqnarray*}
   \tr_{A_s/\rk}(a \overline{\la u,v\ra_{E_s}}\,)& =&\tr_{A_s/\rk}( \bar a \la u,v\ra_{E_s})\\
   &=& \tr_{A/\rk}(\la  \bar a. u,v\ra_E)\\
&  =&\tr_{A/\rk}(\overline{\la  \bar a. u,v\ra_E}\,)\\
   &=& \tr_{A/\rk}((-1)^{i+j}\la  a. u,v\ra_E)\\
  &= &\tr_{A_s/\rk}((-1)^{i+j} a \la u,v\ra_{E_s}).
\end{eqnarray*}
Therefore $\overline{\la u,v\ra_{E_s}}=(-1)^{i+j}  \la u,v\ra_{E_s}$ and the lemma follows.
\end{proof}

By Lemma \ref{hgd}, $E_s$ is a graded Hermitian $A_s$-module. We call it a Harish-Chandra descent of $E$.

We say that a graded Hermitian $A$-module is complex if it is so as a graded $A$-module.

\begin{lemd}\label{hcc}
Assume that $s\in \v(E)$ and  $E$ is complex. Then the Harish-Chandra descent $E_s$ of $E$ is also complex.
\end{lemd}
\begin{proof}
Without loss of generality, we assume that $A$ is simple. If $A_s$ is complex, then $E_s$ is complex by Lemma \ref{a0e}.
 Using Lemma \ref{as}, we assume that  $A_s=A\times A'$ for some complex graded involutive algebra $A'$.
Note that $A'\otimes_{A_s} E_s$  is complex as a graded $A'$-module (Lemma \ref{a0e}). Then by the equality
 \be\label{comdec}
 E_s=(A\otimes_{A_s} E_s)\times (A'\otimes_{A_s} E_s),
\ee
it suffices to show that $A\otimes_{A_s} E_s$ is complex. Note that both $E_s$ and $A'\otimes_{A_s} E_s$ are complex as graded $A$-modules.
Thus by Lemma \ref{compcri},  \eqref{comdec}
implies  that the graded $A$-module $A\otimes_{A_s} E_s$ is complex.  This proves the lemma.
\end{proof}

Similarly,  we have the following  result for  $s\in \oV(E)$.
\begin{lemd}\label{cimpliesc} Assume that $E$ is complex and $s=x\bar{x}^{-1}$ for some  $x\in \oG(E)$ such that $x$ commutes with $\bar{x}$.
Then the Harish-Chandra descent $E_s$ of $E$ is also complex.
\end{lemd}
\begin{proof} As in the proof of  Lemma \ref{hcc}, we assume without loss of generality that $A$ is simple. If $A_s$ is complex, then the lemma follows by Lemma \ref{a0e}. If  $A_s$ is the product of a complex graded involutive algebra and $A$, then
the lemma follows by the same proof as in Lemma \ref{hcc}. Thus by Lemma \ref{as2}, we may (and do) further assume that
$A_s=A_+\times A_-\times A'$, where $A'$ is complex, $A_\pm=A$ and the image of
$s$ via the projection $A_s\rightarrow A_\pm$ is $\pm 1$.

Write $E_\pm:=A_\pm\otimes_{A_s} E_s$ and $E':=A'\otimes_{A_s} E_s$. Then we have that
\[E_s=E_+\times E_-\times E'\quad \text{and} \quad \oG(E_s)=\oG(E_+)\times \oG(E_-)\times \oG(E').\]
Note that $x\in \oG(E_s)\subset \oG(E)$.
Write $x_-$ for the image of $x$ under the projection $\oG(E_s)\rightarrow \oG(E_-)$.
Then the equality
\[x_-\bar{x}_-^{-1}=-1\]
 implies that $x_-$ exchanges $(E_-)_0$ and $(E_-)_1$. Thus  $E_-$ is complex.
Note that $E'$ is also complex (Lemma \ref{a0e}). Thus it suffices to prove that $E_+$ is complex.
Indeed, we know that  $E_s,E_-$ and $E'$  are all complex as graded  $A$-modules.
 By Lemma \ref{compcri}, this implies that $E_+$ is also complex, as required.
\end{proof}
\subsection{Complex Hermitian modules over split graded involutive algebras}
Note that every involutive algebra is the product of all its simple involutive quotients (an involutive quotient is a quotient by a $\tau$-stable ideal), and that every simple involutive algebra is either a field, or
the product of two fields which are exchanged by the involutive automorphism.
\begin{dfn}
We say that $A$ is split if every simple involutive quotient of $A_0$ is
the product of two fields which are exchanged by the involutive automorphism.
\end{dfn}

Let $A\rightarrow A'$ be a homomorphism of graded involutive algebras. If $A$ is split, then $A'$ is also split.
In particular, we get the following lemma.
\begin{lemd}\label{sis}
The Harish-Chandra descent of a split graded involutive
algebra is also split.
\end{lemd}

 Let $\rk'$ be a field extension  of $\rk$ of finite degree. With the  coordinate exchanging automorphism, $\rk'\times \rk'$ is obviously a simple, real, split graded involutive algebra.
  Let $\rk''$ be a  quadratic separable algebra over $\rk'$.
 It is thus either a quadratic field extension of $\rk'$, or a product of two copies of $\rk'$.
  We view $\rk''$ as a graded algebra so that its degree $0$ subalgebra equals $\rk'$.   Then $\rk''\times \rk''$ is also a graded algebra.
   Together with the  coordinate exchanging automorphism, $\rk''\times \rk''$ becomes a graded involutive algebra which is simple, split and complex.
    Conversely, we have the following elementary lemma whose proof is omitted.

   \begin{lemd}\label{simplerealcomplex}
Every  real, simple, split  graded involutive algebra has the form $\rk'\times \rk'$ as above; and
every  complex, simple, split graded involutive algebra has the form $\rk''\times \rk''$ as above.
\end{lemd}

Only complex  graded Hermitian modules over split graded involutive algebras will appear in the proof of Theorem \ref{uniquef}.
Thus, in the rest part of this paper, we assume that
\begin{itemize}
\item
the graded involutive algebra $A$  is split, and the graded Hermitian $A$-module $E$
is  complex.
\end{itemize}

Fix an element $\alpha\in \oV(A)$.

\begin{lemd}\label{beta0}
If $A$ is  complex, then there is an element $\beta\in A^\times$ such that
\[
 \beta \beta^\tau =1 \quad\textrm{and}\quad {\beta}{\bar \beta}^{-1}=\alpha.
\]
\end{lemd}
\begin{proof}
Assume that $A$ is simple without loss of generality.  Write $A=\rk''\times \rk''$ as in Lemma \ref{simplerealcomplex}. Then the lemma is a reformulation of Hilbert Theorem 90.
\end{proof}

\begin{lemd}\label{betais} If $A$ is   complex and  $\beta$ is as in Lemma \ref{beta0}, then the map
\be\label{isog0}
\begin{array}{rcl}
 \breve \oH_{\alpha}(E)&\rightarrow &\breve \oG(E_0),\\
  (g,\delta)&\mapsto & \left\{
                         \begin{array}{ll}
                           (g|_{E_0}, 1), & \hbox{if $\delta=1$;}\smallskip \\
                           ((\beta g)|_{E_0}, -1), & \hbox{if $\delta=-1$}
                         \end{array}
                       \right.
  \end{array}
\ee
is a well-defined group isomorphism.
\end{lemd}

\begin{proof}
Note that $1\in \oV(A)$, and the map
\bee\label{isog2}
\begin{array}{rcl}
 \breve \oH_{\alpha}(E)&\rightarrow &\breve \oH_1(E),\\
  (g,\delta)&\mapsto & \left\{
                         \begin{array}{ll}
                           (g, 1), & \hbox{if $\delta=1$;}\smallskip \\
                           (\beta g, -1), & \hbox{if $\delta=-1$}
                         \end{array}
                       \right.
  \end{array}
\eee
is a well-defined group isomorphism. Therefore, in order to prove the lemma, we may (and do) assume that $\alpha=\beta=1$. Then it is clear that \eqref{isog0} is a group homomorphism. It is bijective since it has an inverse map
\bee\label{isog2}
\begin{array}{rcl}
 \breve \oG(E_0)&\rightarrow &\breve \oH_1(E),\\
  (g,\delta)&\mapsto & \left\{
                         \begin{array}{ll}
                           (1_A\otimes g, 1), & \hbox{if $\delta=1$;}\smallskip \\
                           (\tau\otimes g, -1), & \hbox{if $\delta=-1$.}
                         \end{array}
                       \right.
  \end{array}
\eee
\end{proof}

If $A=\rk'\times \rk'$ is real and simple as in Lemma \ref{simplerealcomplex}, then
\be\label{ohe}
  \oH(E)=\oG(E_0)\times \oG(E_1)\cong \GL_n(\rk')\times \GL_n(\rk')\quad (\textrm{see \eqref{gln}}),
\ee
  where $n:=\rank_A(E_0)=\rank_A(E_1)$. Moreover,
 \begin{equation}\label{bgc}
 \breve \oH_1(E)=\breve \oG(E_0)\times_{\{\pm 1\}} \breve \oG(E_1)\quad(\textrm{the fiber product}),
\end{equation}
and
  \begin{equation}\label{gm1}
\breve \oH_{-1}(E) \cong \{\pm 1\}\ltimes (\GL_n(\rk')\times \GL_n(\rk')),
 \end{equation}
                 where  the semidirect product is defined by the action
              \[
                (-1). (g_1, g_2)=(g_2^{-t}, g_1^{-t}), \quad g_1, g_2\in \GL_n(\rk_0).
              \]

              \begin{lemd}\label{realm1}
Assume that $A$ is real. Then up to conjugation by $\oH(E)\subset \breve \oH_{-1}(E)$,  there exists a unique element of order $2$ in $ \breve \oH_{-1}(E)\setminus  \oH(E)$.
\end{lemd}
\begin{proof}
Without loss of generality assume that $A$ is simple. Then the lemma easily follows by the isomorhhism \eqref{gm1}.
\end{proof}

 Note that if $A$ is real and simple, then $\alpha=\pm 1$. Combining \eqref{bgc},  Lemma \ref{betais}  and Lemma \ref{realm1},
 we  obtain the following result.
   \begin{prpd}\label{mvwg} The group $\breve\oH_\alpha(E)$ contains $\oH(E)$ as a subgroup of index $2$.
\end{prpd}

\subsection{Some characters}\label{secach}

If $A$ is real and simple, then $\oH(\oE)=\oG(E_0)\times \oG(E_1)$, which is the product of two copies of a general linear group as in \eqref{ohe}. We thus define the notion of good characters of $\oH(E)$ as in the Introduction.
In general, we make the following definition.

\begin{dfn}\label{rele1} A character of $\oH(E)$ is said to be good if its restriction to $\oH(A'\otimes_A E)$ is good, for all real simple graded involutive quotient $A'$ of $A$ . \end{dfn}

Let $\alpha\in \oV(A)$ be as before.
\begin{lemd} \label{singlex}
The set
\begin{equation}\label{x}
 \{x\in \oG(E)\mid x=\alpha \bar x\}
\end{equation}
is a single left $\oH(E)$-coset as well as a single right $\oH(E)$-coset.

\end{lemd}
\begin{proof}
It is routine  to check that  the left translation (and the right translation) of $\oH(E)$ on  the set \eqref{x} is transitive. Thus it remains to show that this set is non-empty. Without loss of generality assume that $A$ is simple. If $A$ is complex, then a scalar multiplication provided by Lemma \ref{beta0} is an element of  the set  \eqref{x}. The case when $A$ is real is obvious.
\end{proof}

With Lemma \ref{singlex} in mind, we make the following definition.
\begin{dfn}\label{rele1}
 A character $\breve \chi$ of $\breve \oH_{\alpha}(E)$ is said to be linearly good if there is a good character $\chi$ of $\oH(E)$ such that for some (and hence all) $x$ in the set \eqref{x},
 \begin{equation}\label{deflg}
  \breve \chi(g)=\chi(xgx^{-1}) \chi(g^{-1})\quad \textrm{for all }g\in \oH(E).
\end{equation}
\end{dfn}

As in the proof of Lemma \ref{cimpliesc}, write
\begin{equation}\label{deca}
A=A'\times A^+\times A^-
\end{equation}
 as a product of graded involutive algebras such that $A'$ is complex, $A_+$ and $A_-$ are real, and
the image of $\alpha$ under the projection map $A\rightarrow A^\pm$  is $\pm 1$. Then
\begin{equation}\label{dece}
  E=E'\times E^+\times E^-,
\end{equation}
where $ E'=A'\otimes_A E $ is a graded Hermitian $A'$-module, and $ E^\pm=A^\pm \otimes_A E $ is a graded Hermitian $A^\pm$-module.

\begin{lemd}\label{good1}
Every linearly good character of $\breve \oH_{\alpha}(E)$ has trivial restriction to $\oH(E')\times \oH(E^+)$.
\end{lemd}
\begin{proof}
Using Lemma \ref{beta0}, we assume that the element $x$ in \eqref{deflg} is   a scalar multiplication when restricted to $E'$. Then the lemma easily follows.
\end{proof}

 \begin{lemd}\label{good2}
If $A=\rk'\times \rk'$ is simple and real and $\breve \oH_{-1}(E)$ is realized as in \eqref{gm1}, then a character of  $\breve \oH_{-1}(E)$ is linearly good if and only if its restriction to $ \oH(E)$ has the form $\gamma\otimes \gamma^{-1}$, where $\gamma$ is a good character of $\GL_n(\rk')$.
\end{lemd}
\begin{proof}
This is elementary and we omit the details.
\end{proof}

\begin{dfn}\label{rele2}
A character $\breve \chi$ of $\breve\oH_\alpha(E)$ is said to be linearly  relevant if $\breve \chi(g)=-1$ for every element $g\in \breve \oH_{\alpha}(E)\setminus  \oH(E)$ whose image under the  obvious homomorphism $ \breve\oH_{\alpha}(E)\rightarrow  \breve\oH_{-1}(E^-)$ has order $2$.

\end{dfn}

Note that every linearly relevant character of $\breve\oH_\alpha(E)$ also has trivial restriction to $\oH(E')\times \oH(E^+)$.

In this subsection, let $s$ be a semisimple element of $\v(E)$. Write $\alpha_s$ for the image of $\alpha$ under the
natural embedding $A\hookrightarrow A_s$. Note that $\alpha_s\in \oV(A_s)$ and  $\breve{\oH}_{\alpha_s}(E_s)$ is a subgroup of  $\breve{\oH}_{\alpha}(E)$.

\begin{lemd}\label{stagp1}
Every linearly good character  of  $\breve{\oH}_{\alpha}(E)$ restricts to a linearly good  character of  $\breve{\oH}_{\alpha_s}(E_s)$, and
every linearly relevant character  of  $\breve{\oH}_{\alpha}(E)$ restricts to a linearly relevant character of  $\breve{\oH}_{\alpha_s}(E_s)$.
\end{lemd}
\begin{proof}
The first assertion is obvious since every good character of $\oH(E)$ restricts to a good character of $\oH(E_s)$.
Note  that the  decomposition  \eqref{dece} is $A_s$-stable, and
\[
A_s=(A')_{s'}\times (A^+)_{s^+}\times (A^-)_{s^-},
\]
where $s'\in \v(E')$ is the restriction of $s$ to $E'$, and  $s^\pm \in \v(E^\pm )$ is the restriction of $s$ to $E^\pm$.  The second assertion of the  lemma then easily follows by the commutative diagram
 \[
 \begin{CD}
         \breve{\oH}_{\alpha}(E) @>>> \breve{\oH}_{-1}(E^-)\\
            @AA A           @AAA \\
         \breve{\oH}_{\alpha_s}(E_s)@>> >\breve{\oH}_{-1}((E^-)_{s^-}). \\
  \end{CD}
\]

\end{proof}

\subsection{Some characters on a doubling group}
We  form the semi-direct product
\begin{equation*}\label{semi1}
\breve \BG(E):=  \{\pm 1\}\ltimes(\breve{\oG}(E)\times\breve{\oG}(E))
\end{equation*}
by letting $\{\pm 1\}$ act on $\breve{\oG}(E)\times \breve{\oG}(E)$
as
\[
  (-1).(\breve g,\breve h):=(\breve h,\breve g),\qquad \breve g, \breve h\in \breve{\oG}(E).
\]
Set $\breve\oH(E):=\breve\oH_1(E)$ and consider the fiber product
\[
 \breve{\BH}(E):=\{\pm 1\}\ltimes_{\{\pm 1\}}(\breve{\oH}(E)\times _{\{\pm 1\}}
\breve{\oH}(E))=\{(\delta,g,h)\mid (g,\delta), (h,\delta)\in\breve{\oH}(E)\}.
\]
It is a subgroup of $\breve \BG(E)$, and  contains $\oH(E)\times
\oH(E)$ as a subgroup of index two.

Parallel to Definition \ref{rele1}, we make the following definition.
\begin{dfn}\label{rele3}
A character of $\breve{\BH}(E)$ is said to be doubly  good   if its restriction to $\oH(E)\times \oH(E)$ equals $\chi\otimes \chi^{-1}$ for some good character $\chi$ of $\oH(E)$.
\end{dfn}

Parallel to Definition \ref{rele2}, we make the following definition.
\begin{dfn}\label{rele4}
A character $\breve \xi$ of $\breve{\BH}(E)$ is said to be doubly relevant  if
\[
  \breve \xi(\delta, g,g)=\delta\quad \textrm{for all }(g, \delta)\in \breve{\oH}(E).
\]
\end{dfn}

Let $x$ be an element of $\oG(E)$ which is normal in the sense of \cite{AG}, namely, $x \bar x=\bar x x$. In this subsection, put
\begin{equation*}\label{vE0}
  s:= x \bar{x}^{-1}=\bar{x}^{-1} x\in \oV(E),
\end{equation*}
and assume it is semisimple as a $\rk$-linear operator on $E$. Note that $s\in \oV(A_s)$.
Define a map
\begin{equation}\label{adx}
\begin{array}{rcl}
  \jmath_x: \breve \oH_s(E_s)&\rightarrow& \breve \BH(E),\\
   (g,\delta)&\mapsto&
   \left\{\begin{array}{ll}
     (1, xgx^{-1}, g),\quad &\textrm{if $\delta=1$};\smallskip\\
     (-1, gx^{-1}, xg),\quad &\textrm{if $\delta=-1$}.\\
    \end{array}\right.
    \end{array}
\end{equation}
This is a well-defined group homomorphism.

We prove the following proposition in the rest of this subsection.
\begin{prpd}\label{pad0}
Let $\breve \xi$ be a character on $\breve \BH(E)$. If $\breve \xi$ is doubly relevant or doubly good, then the character $\breve \xi\circ \jmath_x$  of $\breve \oH_s(E_s)$ is respectively linearly relevant  or linearly good.
\end{prpd}

Note that $x\in \oG(E_s)$, the image of the map \eqref{adx} is contained in $\breve \BH(E_s)$, and every doubly good or doubly relevant character of $ \breve \BH(E)$ restricts to a character of $\breve \BH(E_s)$ which is respectively doubly good or doubly relevant. Thus for the proof of Proposition \ref{pad0}, we assume without loss of generality that $s=\alpha\in A$.

Write
\[
A=A'\times A^+\times A^-\quad\textrm{and}\quad E=E'\times E^+\times E^-,
\]
as in \eqref{deca} and \eqref{dece}.

\begin{lemd}\label{g2}
Let $(g, -1)\in \breve \oH_\alpha(E)$. Assume that the image of $(g, -1)$ under the natural homomorphism $ \breve \oH_\alpha(E)\rightarrow  \breve \oH_{-1}(E^-)$ has order $2$. Then there is an element $(b, -1)\in \breve \oH(E)$ such that $b^2=g^2$.
\end{lemd}
\begin{proof}
Without loss of generality assume that $A$ is simple. The lemma is obvious when $A$ is real. So we further assume that $A$ is complex.
 Using Lemma \ref{beta0}, take an element $\beta\in {A}^\times$  such that
\[
 \beta \beta^\tau =1 \quad\textrm{and}\quad {\beta}{\bar \beta}^{-1}=\alpha.
\]
Then $b:=\beta g$
fulfills the requirement of the lemma.
\end{proof}

Let $\breve \xi$ be a character on $\breve \BH(E)$ as in Proposition \ref{pad0}.
\begin{lemd}\label{pad}
If $\breve \xi$ is doubly  relevant,  then character $\breve \xi\circ \jmath_x$  is linearly  relevant.

\end{lemd}

\begin{proof}
Let $(g, -1)$ be as in Lemma \ref{g2}.  Then $(-1, gx^{-1}, b)\in \breve \BH(E)$ and
\[
  (-1, gx^{-1}, b) (-1, gx^{-1}, xg) (-1, gx^{-1}, b)^{-1}=(-1, b,b),
\]
where
$b$ is as in Lemma \ref{g2}. The lemma then easily follows.
\end{proof}

It is obvious that if  $\breve \xi$ is doubly  good,  then the character $\breve \xi\circ \jmath_x$  is linearly good. This finishes the proof of Proposition \ref{pad0}.

\section{A vanishing result of generalized functions}
As before, let $A=A_0\oplus A_1$ be a split graded involutive algebra, $\alpha\in \oV(A)$, and let $E=E_0\oplus E_1$ be a complex graded Hermitian $A$-module.
\subsection{The main result}\label{svar}

Let the group $\breve{\oH}_\alpha(E)$ act on $\v(E)$ by
\begin{eqnarray*}\label{gactionv}
  (g,\delta).x:=\delta gxg^{-1},\qquad (g,\delta)\in \breve{\oH}_\alpha(E),\ x\in \v(E).
\end{eqnarray*}
The main goal of this section is to prove the following result.

\begin{thmd}\label{vanishl}
Let $\breve \chi$ be a character of $\breve{\oH}_\alpha(E)$ which is linearly good and linearly relevant. Then the space of ${\breve \chi}$-equivariant generalized functions on
$\v(E)$ is zero, that is,
\begin{equation}\label{vanishl2}
  \con^{-\infty}_{\breve \chi}(\v(E))=0.
\end{equation}
\end{thmd}

Recall that a generalized function $f$ on $\v(E)$ is said to be ${\breve \chi}$-equivariant if for all $g\in \breve{\oH}_\alpha(E)$,
\[f(g.x)={\breve \chi}(g)f(x),\quad  x\in \v(E).\]
The space of such generalized functions is denoted by $\con^{-\infty}_{\breve \chi}(\v(E))$. Similar notation  will be used later on without further explanation.

 Let the group $\breve{\oG}(E_0)$ act on the Lie algebra
$\g(E_0)$ by
\begin{equation*}\label{actiong}
  (g,\delta).x:=\delta gxg^{-1},\quad (g,\delta)\in \breve{\oG}(E_0), \, x\in \g(E_0).
\end{equation*}

\begin{lemd}\label{beta} Assume that $A$ is complex. Then there is an element $\gamma$  of $A_1\cap A^\times$ such that $\gamma^\tau=\gamma$. Moreover, the map
\be\label{isov}
  \v(E)\rightarrow \g(E_0),\quad x\mapsto (\gamma x)|_{E_0}
\ee
is a well-defined $\rk$-vector space isomorphism which is equivariant with respect to the group isomorphism $ \breve \oH_{\alpha}(E)\rightarrow \breve \oG(E_0)$ of  \eqref{isog0}.
\end{lemd}

\begin{proof}
The existence of such $\gamma$  follows form   Lemma \ref{simplerealcomplex}.
It is routine to check that the map \eqref{isov} is well defined and equivariant with respect to the group isomorphism  \eqref{isog0}.
It is bijective since it has an inverse map
\[
  \g(E_0)\rightarrow \v(E),\quad x\mapsto \gamma^{-1}(1_A\otimes x).
\]
\end{proof}

Theorem \ref{vanishl} is easily reduced to the case when $A$ is simple. When $A$ is complex, we know from Lemma \ref{beta} that Theorem \ref{vanishl}  is equivalent to saying that
\begin{equation*}\label{vamvw}
\con^{-\infty}_{{\breve \chi}_{E_0}}(\g(E_0))=0,
\end{equation*}
where ${\breve \chi}_{E_0}$ is the quadratic character of $\breve{\oG}(E_0)$ with kernel $\oG(E_0)$.  This is a reformulation of part (a) of Theorem \ref{tlin}, which is well known.
We record this result in  the following proposition.

\begin{prpd}\label{vanishcom}
Theorem \ref{vanishl} holds when $A$ is complex.
 \end{prpd}

When $A$ is real and $\alpha=1$,  Theorem \ref{vanishl}  is a reformulation of part (b) of Theorem \ref{tlin}. When $A$ is real and $\alpha=-1$,  Theorem \ref{vanishl}  is a reformulation of part (c) of Theorem \ref{tlin}.



\subsection{The Fourier transform}
Let $\breve \chi$ be as in Theorem \ref{vanishl}.
Recall that by \cite[Theorem 4.2]{AG}, the equality \eqref{vanishl2}
is implied by
\begin{equation}\label{vanishl2'}
  \con^{-\xi}_{\breve \chi}(\v(E))=0.
\end{equation}
Here the left-hand side of \eqref{vanishl2'} stands for the space of
${\breve \chi}$-equivariant tempered generalized functions on
$\v(E)$, and similar notation will be used later on. Note that in the non-archimedean case, all generalized functions are said to be tempered by convention.

 Define a non-degenerate symmetric $\rk$-bilinear form on $\g(E)$ by
\begin{equation}\label{bform}
  \la y,z\ra_{\g(E)}:=\textrm{ the trace of $yz$ as a $\rk$-linear operator
  on $E$.}
\end{equation}
Note that the restriction of this bilinear form on $\v(E)$ is still non-degenerate.
Fix a non-trivial unitary character  $\psi_\rk$  of $\rk$ as in the Introduction. Denote by
\begin{equation}\label{fourier}
\mathcal F:\con^{-\xi}(\v(E))\rightarrow \con^{-\xi}(\v(E))
\end{equation}\label{fourt}
 the Fourier transform which is normalized such that for every Schwartz function $f$ on $\v(E)$,
\begin{equation}\label{fourt00}
\mathcal F(f)(x)=\int_{\v(E)}f(y)\psi_\rk(\la x,y\ra_{\g(E)}) \mathrm{d}\,y,\quad x\in \v(E),
\end{equation}
where $\mathrm{d} y$ is the self-dual Haar measure on $\v(E)$. It is clear that  the Fourier transform \eqref{fourier} intertwines the action of $\breve{\oH}_\alpha(E)$.  Thus we have the following lemma.

\begin{lemd}\label{four0}
The Fourier transform $\CF$ preserves the space $\con^{-\xi}_{\breve \chi}(\v(E))$.
\end{lemd}

\subsection{Reduction to the null cone}

Set
 \[\CN_E:=\{x\in \v(E)\mid x \text{\ is nilpotent as a $\rk$-linear operator on}\ E\}\]
and
\[\mathrm{sdim}(E):=\dim_\rk(E)-\dim_\rk(A).\]
We shall prove the following proposition in this subsection.

 \begin{prpd}\label{reduction}
 Assume that for all split graded involutive algebra $A'$, all $\alpha'\in \oV(A')$, all faithful complex graded Hermitian $A'$-module
$E'$ and all
character ${\breve \chi}'$ on $\breve{\oH}_{\alpha'}(E')$ which are linearly good and linearly relevant,
\begin{equation}\label{indu}
\mathrm{sdim}(E')<\mathrm{sdim}(E)\quad \Rightarrow \quad \con^{-\xi}_{{\breve \chi}'}(\v(E'))=0.
\end{equation}
Then every $f\in \con^{-\xi}_{\breve \chi}(\v(E))$ is supported in $\v(A)+\CN_E$, where
\[\v(A):=\{a\in A\mid a^\tau=a\ \text{and }\bar{a}=-a\}\subset \v(E).\]
\end{prpd}

 Fix a semisimple element $s\in \v(E)\setminus\v(A)$.
Then we have that $\dim_\rk(A)<\dim_\rk(A_s)$ and hence
$\mathrm{sdim}(E_s)<\mathrm{sdim}(E)$. Put
\[\v(E_s)^\circ:=\{y\in \v(E_s)\mid J(y)\ne 0\},\]
where $J(y)$ is the determinant of the composition of the following $\rk$-linear  maps
\[
 \v(E)/\v(E_s)\xrightarrow{x\mapsto [x,y]} \h(E)/\h(E_s)\xrightarrow{x\mapsto [x,y] } \v(E)/\v(E_s).
\]
Note that the function  $J$ is $\breve\oH_{\alpha_s}(E_s)$-invariant and thus  $\v(E_s)^\circ$ is a $\breve\oH_{\alpha_s}(E_s)$-stable open subset of $\v(E_s)$, where $\alpha_s$ denotes the image of $\alpha$ under the inclusion map $A\rightarrow A_s$, as before.
Let $\breve{\oH}_\alpha(E)$ act on $\breve{\oH}_\alpha(E)\times \v(E_s)^\circ$ via the left multiplication on the first factor.
Define an $\breve{\oH}_\alpha(E)$-equivariant map
\begin{equation}\label{pis}
\breve{\oH}_\alpha(E)\times \v(E_s)^\circ\rightarrow \v(E),\quad (g,y)\mapsto g.y.
\end{equation}
\begin{lemd}\label{submer}
The map \eqref{pis} is a submersion, and its image contains  $s+\CN_{E_s}$.
\end{lemd}
\begin{proof}
The lemma easily follows from the facts that
\[
  \g(E)=\h(E)\oplus \v(E),
\]
and that the centralizer of $s\in \v(E)$ in $\g(E)$
equals
\[
  \g(E_s)=\h(E_s)\oplus \v(E_s).
\]

\end{proof}

Note that $\breve\oH_{\alpha_s}(E_s)$ equals the stabilizer of $s$ in $\breve\oH_{\alpha}(E)$ under the action \eqref{gactionv}.
Thus the submersion \eqref{pis}  yields a well-defined injective restriction map (\cf \cite[Lemma 2.7]{JSZ})
\[\con_{\breve \chi}^{-\xi}(\v(E))\rightarrow \con_{{\breve \chi}_s}^{-\xi}(\v(E_s)^\circ),\]
where $\breve \chi_s$ denotes the restriction of $\breve \chi$ to  $\breve\oH_{\alpha_s}(E_s)$.
Lemma \ref{stagp1} and assumption \eqref{indu} imply that
\[
\con_{{\breve \chi}_s}^{-\xi}(\v(E_s))=0.
\]
By a standard argument (\cf \cite[Section 5.1]{JR}), this implies that
\[
\con_{{\breve \chi}_s}^{-\xi}(\v(E_s)^\circ)=0.
\]
Thus  every $f\in \con_{\breve \chi}^{-\xi}(\v(E))$ vanishes on the image of \eqref{pis}, which contains $s+\CN_{E_s}$ by Lemma \ref{submer}.
This completes the proof of Proposition \ref{reduction} by the following lemma.

\begin{lemd}
There is a decomposition
\[
\v(E)=\bigsqcup_{s\textrm{ is a semisimple element of $\v(E)$}} (s+\CN_{E_s}).
\]
\end{lemd}
\begin{proof}
This easily follows from the Jordan Decomposition Theorem for the Lie algebra $\g(E)$ of $\oG(E)$.
\end{proof}
\subsection{Reduction within the null cone}
Let $V=V_0\oplus V_1$ be a $\Z/2\Z$-graded finite dimensional vector space over $\rk$ with
\[
n:=\dim V_0=\dim V_1\geq 1.
\]
Put
\[\v:=\Hom(V_1,V_0)\oplus \Hom(V_0,V_1)\ \text{ and }\ \h:=\End(V_0)\oplus \End(V_1),\]
which are the odd and even parts  of the $\Z/2\Z$-graded algebra $\End(V)$, respectively.
Set
\begin{equation*}\label{isoggg}
H:=\GL(V_0)\times \GL(V_1)\cong\GL_n(\rk)\times \GL_n(\rk),
\end{equation*}
 which acts naturally on $\v$.
Denote by
\[
\CN_\v:=\{(x,y)\in \v\mid x\circ y: V_0\rightarrow V_0\textrm{ is a nilpotent operator}\}
\]
 the nilpotent cone in $\v$.
We shall prove the following result in this subsection.

\begin{prpd}\label{reduce2}
Let $\gamma$ be a good character of $\GL_n(\rk)$ as in the Introduction, and view $\gamma\otimes \gamma^{-1}$ as a character of $H$ via the isomorphism in \eqref{isoggg}. Let $f$ be a  $\gamma\otimes \gamma^{-1}$-equivariant tempered
generalized function on $\v$ such that  both $f$ and its Fourier transform $\mathcal F(f)$ are supported in $\CN_\v$. Then $f$ is the zero function.
\end{prpd}

Here the Fourier transform $\CF$ is defined as in \eqref{fourt00}. When $\gamma$ is trivial, Proposition \ref{reduce2}  is proved in \cite{JR} for the non-archimedean case and in \cite{AG} for the
archimedean case. Our proof for Proposition \ref{reduce2} is similar to that in \cite{AG}.

 Write $\s$ for the Lie algebra
 $\s\l_2(\rk)$ equipped with a $\Z/2\Z$-grading $\s=\s_0\oplus \s_1$ such that
\[  \left[ \begin{array}{cc}
              1&0\\ 0&-1
              \end{array}
              \right]\in \s_0\quad \text{and}\quad  \left[\begin{array}{cc}
              0&1\\ 0&0
              \end{array}
              \right],\left[\begin{array}{cc}
              0&0\\ 1&0
              \end{array}
              \right]\in \s_1.\]
A graded $\s$-module is defined to be an $\s$-module $W$ with a  $\Z/2\Z$-grading
$W=W_0\oplus W_1$ such that
\[\s_i.W_j\subset W_{i+j},\quad i,j\in \Z/2\Z.\]
For every non-negative integer $\lambda$ and every $\omega\in \Z/2\Z$, we write $V_\lambda^\omega$ for the
graded $\s$-module such that it is the irreducible highest weight module with highest weight $\lambda$  as a $\s\l_2(\rk)$-module, and
that the highest weight vector has grading $\omega$. Note that the graded $\s$-module $V_\lambda^\omega$ is graded-irreducible, namely it is nonzero and has no nonzero proper  graded submodule.  Conversely, every
graded-irreducible  $\s$-module is isomorphic to $V_\lambda^\omega$ for a uniquely determined pair $(\lambda, \omega)$. Moreover, every graded $\s$-module is a direct sum of graded-irreducible  $\s$-modules.

Let $\mathcal O$ be an $H$-orbit in $\CN_\v$.
Recall that every $\mathbf e\in \mathcal O$ can be extended to a graded $\s\l_2$-triple $\{\mathbf{h,e,f}\}$ in the sense that (\cf \cite[Proposition 4]{KR})
\begin{equation}\label{sl2t}
[\mathbf h,\mathbf e]=2\mathbf e,\ [\mathbf h,\mathbf f]=-2\mathbf f,\ [\mathbf e,\mathbf f]=\mathbf h,\  \mathbf f\in \CN_\v \ \, \text{and}\ \, \mathbf h \in \h.
\end{equation}
Via this triple, $V$ becomes a graded $\s$-module. Decompose this
 graded $\s$-module as
\[V=V_{\lambda_1}^{\omega_1}\oplus V_{\lambda_2}^{\omega_2}\oplus\cdots \oplus V_{\lambda_d}^{\omega_d}, \qquad(d\geq 1).\]


Write
\[\mathbf h=(\mathbf h_0,\mathbf h_1)\in \h=\End(V_0)\times \End(V_1),\]
and set
\[\widehat{\mathbf h}:=(\mathbf h_0,-\mathbf h_1)\in \h.\]
The following lemma is easy to check.
\begin{lemd}\label{constant} For each $i=1,2,\cdots,d$, one has that
\begin{equation*}\tr(\widehat{\mathbf{h}}|_{V_{\lambda_i}^{\omega_i}})=
\begin{cases}
0,\quad &\text{ if $\lambda_i$ is even};\\
\lambda_i+1,\quad &\text{ if $\lambda_i$ is odd and } \omega_i=0;\\
-\lambda_i-1,\quad &\text{ if $\lambda_i$ is odd and } \omega_i=1.
\end{cases}
\end{equation*}
In particular, one has that
\[\tr(\widehat{\mathbf{h}})\in\{0,\pm 2,\cdots,\pm 2n\}.\]
\end{lemd}

For each $1\le i,j\le d$, set
\begin{equation*}\label{comten0}
m_{i,j}:=\tr((2-\mathbf h)|_{\Hom(V_{\lambda_i}^{\omega_i},V_{\lambda_j}^{\omega_j})_1^{\mathbf f}})
+\tr((2-\mathbf h)|_{\Hom(V_{\lambda_j}^{\omega_j},V_{\lambda_i}^{\omega_i})_1^{\mathbf f}})-(\lambda_i+1)(\lambda_j+1).
\end{equation*}
Here  $\Hom(V_{\lambda_j}^{\omega_j},V_{\lambda_i}^{\omega_i})$ is obviously viewed as a graded  $\s$-module, and $\Hom(V_{\lambda_j}^{\omega_j},V_{\lambda_i}^{\omega_i})_1^{\mathbf f}$ is the space of vectors in its odd part which are annihilated by $\mathbf f$. Similar notation will be used without further explanation.

\begin{lemd}\label{comten} For each $1\le i,j\le d$, one has that
\begin{equation*}m_{i,j}=
\begin{cases}
\min\{\lambda_i,\lambda_j\}+1,& \text{ if } \lambda_i\not\equiv \lambda_j\quad (\mathrm{mod}\ 2);\\
2\min\{\lambda_i,\lambda_j\}+2,&\text{ if } \lambda_i\equiv\lambda_j\equiv 1\quad (\mathrm{mod}\ 2) \text{ and }\omega_i=\omega_j;\\
0,& \text{ if } \lambda_i\equiv\lambda_j\equiv 1\quad (\mathrm{mod}\ 2) \text{ and }\omega_i\ne \omega_j;\\
-|\lambda_i-\lambda_j|-1,&\text{ if } \lambda_i\equiv\lambda_j\equiv 0\quad (\mathrm{mod}\ 2) \text{ and }\omega_i=\omega_j;\\
\lambda_i+\lambda_j+3, &\text{ if } \lambda_i\equiv\lambda_j\equiv 0\quad (\mathrm{mod}\ 2) \text{ and }\omega_i\ne\omega_j.
\end{cases}
\end{equation*}
\end{lemd}

\begin{proof}This Lemma is similar to \cite[Lemma 7.7.9]{AG} and its proof is also similar. The numbers $m_{i,j}$  can be
computed directly by the facts that
\begin{equation*}
\tr((2-\mathbf h)|_{(V_\lambda^\omega)_1^{\mathbf f}})=
\begin{cases} \lambda+2,\ &\text{if}\ \lambda+\omega \text{ is odd};\\
0,\ &\text{if}\ \lambda+\omega \text{ is even},
\end{cases}
\end{equation*}
and that
\[(V_{\lambda_i}^{\omega_i})^*\otimes V_{\lambda_j}^{\omega_j}
=\oplus_{l=0}^{\min\{\lambda_i,\lambda_j\}} V_{\lambda_i+\lambda_j-2l}^{\omega_i+\lambda_i+\omega_j-l}.
\]
\end{proof}

Under the adjoint action of the triple $\{\mathbf h,\mathbf e,\mathbf f\}$,
$\End(V)$ becomes a graded $\s$-module with $\v$ as its odd part.
The following result is similar to \cite[Lemma 3.1]{JR} and \cite[Lemma 7.7.5]{AG}.
\begin{lemd}\label{sinequ} One has that
\[2n^2<\tr((2-\mathbf h)|_{\v^{\mathbf f}})\le 4n^2.\]
\end{lemd}
\begin{proof}The proof of this inequality
 is the same as that of \cite[Lemma 7.7.5]{AG} by using Lemma \ref{comten}.
\end{proof}

Let $\gamma$ be a character of $\GL_n(\rk)$ and let $\gamma_\rk$ be  the character  of $\rk^\times$ such that $\gamma=\gamma_\rk\circ \det$. View $\gamma\otimes \gamma^{-1}$ as a character of $H$. Denote by  $\con^{-\xi}(\v,\CO)$ the space of
tempered generalized functions on  $\v\setminus(\partial \CO)$ with support in $\CO$, and denote by
$\con^{-\xi}_{\gamma\otimes \gamma^{-1}}(\v,\CO)$ its subspace of   $\gamma\otimes \gamma^{-1}$-equivariant
elements, where $\partial \CO$ denotes the complement of $\CO$ in its closure in $\v$. We will use similar notation without further explanation.

Let $\rk^\times$ act on $\con^{-\xi}(\v)$ by
\begin{equation}\label{dial}
  (t.f)(x,y)=f(t^{-1}x, t^{-1}y),\quad t\in \rk^\times, \,f \in \con^{-\xi}(\v).
\end{equation}
Note that the orbit $\CO$ is invariant under dilation, and thus  $\rk^\times$ acts on  $\con^{-\xi}_{\gamma\otimes \gamma^{-1}}(\v,\CO)$
as in \eqref{dial}.

\begin{lemd}\label{eigest}
 Let $\eta:\rk^\times\rightarrow \C^\times$ be an eigenvalue for the action of $\rk^\times$ on $\con^{-\xi}_{\gamma\otimes \gamma^{-1}}(\v,\mathcal O)$. Then
 \[
 \eta^2=\gamma_\rk^{-\tr(\widehat{\mathbf h})}\cdot |\\ \cdot\\ |^{\tr\left((2-\mathbf h)|_{\v^{\mathbf f}}\right)}\cdot \kappa\]
for some pseudo-algebraic character $\kappa$ of $\rk^\times$.
\end{lemd}
\begin{proof}
 View $\v$ as an $H\times \rk^\times$-space. Then $\CO$ is an $\oH\times\rk^\times$-orbit
and the $\eta$-eigenspace in $\con^{-\xi}_{\gamma\otimes \gamma^{-1}}(\v,\mathcal O)$ equals $\con^{-\xi}_{(\gamma\otimes \gamma^{-1})\otimes\eta^{-1}}(\v,\CO)$.
The $\s\l_2$-triple $\{\mathbf{h,e,f}\}$ integrates to an algebraic homomorphism
\[\phi: \SL_2(\rk)\rightarrow \GL(V)\]
which maps $\left[ \begin{array}{cc}
              t&0\\ 0&t^{-1}
              \end{array}
              \right]$ to an element, say $D_t$, of $H$.
Set
\begin{equation*}
  T:=\{(D_t,t^{-2})\in H\times \rk^\times\mid t\in \rk^\times\},
  \end{equation*}
which fixes the element $\mathbf e$ and stabilizes the space $\v^\mathbf f$.

By using the equality
\[\v=[\h,\mathbf e]\oplus \v^{\mathbf f},\]
we know that the map
\be\label{submer0}
  (H\times \rk^\times)\times \v^{\mathbf f}\rightarrow \v,\quad (g,v)\mapsto g.(v+\mathbf e)
  \ee
is submersive at every point of  $(H\times \rk^\times)\times \{0\}$, and $(H\times \rk^\times)\times \{0\}$ is open in the inverse image of $\CO$ under  the map \eqref{submer0}. Thus the restriction map yields an injective linear map (\cf \cite[Lemma 2.7]{JSZ} and  \cite[Lemma 5.4]{SZ2})
\[
 \con^{-\xi}_{(\gamma\otimes \gamma^{-1})\otimes\eta^{-1}}(\v,\CO)\rightarrow \con^{-\xi}_{((\gamma\otimes \gamma^{-1})\otimes\eta^{-1})|_T}(\v^{\mathbf f},\{0\}).
 \]
It is easy to see that the representation $\con^{-\xi}(\v^{\mathbf f},\{0\})$ of $T$ is completely reducible and every eighenvalue has the form
\[
(D_t, t^{-2})\mapsto   |t |^{\tr((\mathbf h-2)|_{\v^{\mathbf f}})} \kappa(t^{-1}), \quad t\in \rk^\times
\]
where $\kappa$ is a pseudo-algebraic character of $\rk^\times$. Thus the character $((\gamma^{-1}\otimes \gamma)\otimes\eta)|_T$ has this form, or equivalently,
\[
  \gamma_\rk^{-\tr(\widehat{\mathbf h})}\cdot \eta^{-2}=|\\ \cdot \\ |^{\tr((\mathbf h-2)|_{\v^{\mathbf f}})} \cdot \kappa^{-1}
\]
for some pseudo-algebraic character $\kappa$ of $\rk^\times$. This proves the lemma.

\end{proof}

Note that $\v$ is a split symmetric bilinear space under the trace form, and the associated quadratic form is
\[
  Q(x,y):=\tr(x\circ y)+\tr(y\circ x), \qquad (x,y)\in \v=\Hom(V_1,V_0)\oplus \Hom(V_0,V_1).
\]
Denote by $Z(Q)$ the zero locus of $Q$ in $\v$. Then $\CN_\v\subset Z(Q) \subset \v$.
Recall the following homogeneity result on tempered generalized functions (\cf  \cite[Theorem 5.1.7]{AG}).

\begin{prpd}\label{homcri} Let $L$ be a non-zero subspace of $\con^{-\xi}(\v,Z(Q))$ such that
for every $f\in L$, one has that $\mathcal F(f)\in L$ and $(\psi\circ Q)\cdot f\in L$ for all unitary character $\psi$ of $\rk$.
Then   $L$ is a completely reducible  $\rk^\times$-subrepresentation of $\con^{-\xi}(\v)$, and it has an  eigenvalue of the form
\begin{equation*}\label{etal}
\kappa^{-1}\cdot |\\ \cdot \\|^{\dim\v/2},
\end{equation*}
where $\kappa$ is a pseudo-algebraic character of $\rk^\times$.
\end{prpd}

Now we are prepared  to prove Proposition \ref{reduce2}. Assume that $\gamma$ is good as in Proposition \ref{reduce2}. Denote by $L_\gamma$ the
space of all tempered generalized functions $f$ on $\v$ with the properties as in Proposition \ref{reduce2}.
Assume by contradiction  that $L_\gamma$ is non-zero. Then by Proposition \ref{eigest} and Proposition \ref{homcri}, there is an $\sl_2$-triple $\{\mathbf{h,e,f}\}$ as in \eqref{sl2t} such that
\begin{equation*}
\kappa_1\cdot \gamma_\rk^{-\tr(\widehat{\mathbf h})}\cdot | \\ \cdot \\ |^{\tr((2-\mathbf h)|_{\v^{\mathbf f}})} =\kappa_2^{-2}\cdot | \\ \cdot \\ |^{\dim\v}
\end{equation*}
for some
pseudo-algebraic characters $\kappa_1$ and $\kappa_2$ of $\rk^\times$.
Thus, by Lemma \ref{constant} and
Lemma \ref{sinequ}, there exists $r\in \{0,\pm 2,\cdots,\pm 2n\}$ and $m\in \{1,2, \cdots,  2n^2\}$ such that
\begin{equation}\label{etark}
\gamma_\rk^r=\kappa\cdot  | \\ \cdot \\ |^m
\end{equation}
for some    pseudo-algebraic character $\kappa$ of $\rk^\times$. Note that the equality \eqref{etark} does not hold for $r=0$. Thus $\gamma$ is not a good character and we arrive at a contradiction. Then the space $L_\gamma$ is  zero and we finish the proof of Proposition \ref{reduce2}.

\subsection{Proof of Theorem \ref{vanishl}}
Note that $\mathrm{sdim}(E)\geq 0$ since $E$ is assumed to be faithful as an $A$-module, and the equality holds only when $A$ is complex.
Thus Theorem \ref{vanishl} holds when $\mathrm{sdim}(E)=0$ by Proposition \ref{vanishcom}. Now assume that $\mathrm{sdim}(E)>0$ and Theorem \ref{vanishl} holds when $\mathrm{sdim}(E)$ is smaller. Theorem \ref{vanishl} is easily reduced to the case when $A$ is simple. Together with Proposition \ref{vanishcom}, we may (and do) assume that $A$ is simple and real. Without loss of generality we further assume that  $A=\rk\times \rk$.
Then it follows from Proposition \ref{reduction} that
every element of  $\con^{-\xi}_{\breve \chi}(\v(E))$ has support in $\CN_E$
(the space $\v(A)$ in Proposition \ref{reduction} is zero when $A$ is real).  Together with Lemma \ref{good1}, Lemma \ref{good2}, Lemma \ref{four0} and Proposition \ref{reduce2}, this implies that every element of $\con^{-\xi}_{\breve \chi}(\v(E))$ is zero.

\section{A proof of Theorem \ref{uniquef}}

Let the group $\breve \BH(E)$ act on $ \oG(E)$ by
\[
   (\delta,\breve g,\breve h). x:=(\breve g  x \breve h^{-1})^{\delta},\qquad  (\delta,\breve g,\breve h)\in \breve \BH(E),\
   x\in \oG(E).
\]
This section is devoted to a proof of the following theorem.

\begin{thmd}\label{vanisht}
Let $\breve \xi$ be a character of $\breve{\BH}(E)$ which is doubly relevant and doubly good. Then
the space of $\breve \xi$-equivariant generalized functions on $\oG(E)$ is
zero, in other words,
\begin{equation}\label{vanish}
  \con^{-\infty}_{\breve \xi}(\oG(E))=0.
\end{equation}
\end{thmd}

If $A=\rk\times \rk$ is real and simple, then Theorem \ref{vanisht}  is just a reformulation of Theorem \ref{uniquef}.

By \cite[Theorem 3.1.1]{AG}, Theorem
\ref{vanisht} is implied by the following assertion:
\begin{equation}\label{van1}
  \con^{-\infty}_{\breve \xi}(\operatorname
N_{O}^{\oG(E)})=0\quad\textrm{for all closed
$\breve{\BH}(E)$-orbits $O\subset \oG(E)$.}
\end{equation}
Here
\[
  \operatorname
N_{O}^{\oG(E)}:=\bigsqcup_{x\in O}\operatorname
N_{O,x}^{\oG(E)}\qquad (\operatorname N_{O,x}^{\oG(E)}:=\operatorname
T_x (\oG(E))/\operatorname T_x O)
\]
is the normal bundle of $O$ in $\oG(E)$. It is naturally an
$\breve{\BH}(E)$-homogeneous vector bundle.

\begin{lemd}
For every closed $\breve{\BH}(E)$-orbit $O\subset \oG(E)$,
there is an element $x\in O$ which is normal in the sense that $x$
and $\bar{x}$ commute with each other.
\end{lemd}
\begin{proof}
By \cite[Corollary 7.7.4.]{AG} and its proof, we know that the
symmetric pair $(\oG(E), \oH(E))$ is ``good" in the sense that
every closed double $\oH(E)$-coset in $\oG(E)$ is stable under the
map $y\mapsto \bar{y}^{-1}$.  Therefore the lemma follows from
\cite[Lemma 7.4.7]{AG}.
\end{proof}

Let $O\subset \oG(E)$ be a closed
$\breve{\BH}(E)$-orbit, and let $x\in O$ be a normal
element so that $x\bar x=\bar x x$. By Frobenious reciprocity (\cf \cite[Theorems 3.3 and
3.4]{AGS}), \eqref{van1} is equivalent to
\begin{equation}\label{van2p}
  \con^{-\infty}_{\breve \xi_x}(\operatorname
N_{O, x}^{\oG(E)})=0.
\end{equation}
Here $\breve \xi_x$ is the restriction of $\breve \xi$ to the stabilizer
$\breve{\BH}_x\subset \breve{\BH}(E)$ of $x$.

Put
\begin{equation*}\label{vE}
   s:= x \bar{x}^{-1}\in \oG(E).
\end{equation*}
Since the orbit $O$ is assumed to be closed, \cite[Proposition
7.2.1]{AG} implies that $s$ is semisimple. Recall the homomorphism
\[
 \jmath_x: \breve \oH_s(E_s)\rightarrow \breve \BH(E)
\]
from \eqref {adx}. This homomorphism is clearly injective and it is routine to check that  its image equals the stabilizer group $\breve{\BH}_x$. We identify $\breve{\BH}_x$ with $\breve \oH_s(E_s)$ via this homomorphism.

Identify the tangent space $\operatorname{T}_x (\oG(E))$ with
$\g(E)=\operatorname{T}_1(\oG(E))$ through the left translation.
Then the isotropic representation of $\breve{\BH}_x$ on
$\operatorname{T}_x (\oG(E))$ is identified with the following representation of $\breve{\oH}_s(E_s)$ on $\g(E)$:
\begin{equation*}\label{isoact}
  (g,\delta).y=\delta gyg^{-1},\qquad (g,\delta)\in
  \breve{\oH}_s(E_s),\,\, y\in \g(E).
\end{equation*}
This representation preserves the non-degenerate bilinear form $\la\, ,\, \ra_{\g(E)}$ (see \eqref{bform}).

\begin{lemd}\label{decomge}
One has a decomposition
\begin{equation*}\label{decomue}
 \g(E)=(\h(E)+\Ad_{x^{-1}}(\h(E)))\oplus \v(E_s)
\end{equation*}
of representations of  $\breve{\oH}_s(E_s)$.
\end{lemd}
\begin{proof}
Note that $\g(E)=\h(E)\oplus \v(E)$ is an orthogonal decomposition with respect to the
bilinear form $\la , \ra_{\g(E)}$. Thus an element $y\in \g(E)$ is perpendicular to
$\h(E)+\Ad_{x^{-1}}(\h(E))$  if and only if both $y$ and $\Ad_x
y$ belong to $\v(E)$, that is,
\[
   \bar{y}=-y\quad\textrm{and}\quad \bar{x}\bar{y}\bar{x}^{-1}=-xyx^{-1}.
\]
This is equivalent to saying that $y\in \v(E_s)$. The lemma then
follows as the space $\v(E_s)$ is non-degenerate.

\end{proof}

Note that the tangent space
\[
  \operatorname{T}_x O=\h(E)+\Ad_{x^{-1}}(\h(E))\subset \g(E)=\operatorname{T}_x (\oG(E)).
\]
Hence by Lemma \ref{decomge}, the normal space
\[
   \operatorname{N}^{\oG(E)}_{O,x}=\frac{\g(E)}{\g(E)+\Ad_{x^{-1}}(\g(E))}\cong \v(E_s)
\]
as  a $\rk$-linear representation of
$\breve{\oH}_s(E_s)$. Thus, in view of Proposition \ref{pad0},
 \eqref{van2p} follows by Theorem \ref{vanishl}, and consequently,
Theorem \ref{vanisht} is proved.

\section{Proof of Theorem \ref{unique0}}\label{proof A}
This short section is devoted to a proof of Theorem \ref{unique0}.
The proof is similar to that in \cite{FJ, JR, AG}, but the consideration of meromorphic continuation is avoided due to the proof of Theorem  \ref{uniquel}.  Let $\pi$ be an irreducible admissible smooth representation of $\GL_{2n}(\rk)$ as in Theorem \ref{unique0}, and let
 $\lambda\in \Hom_{\oS_n(\rk)}(\pi, \psi_{\oS_n})$ (see \eqref{sshalika}). For every $v\in \pi$,  let $\phi_{\lambda, v}$ denote the following function on $\GL_n(\rk)$: \[\phi_{\lambda,v}:\GL_n(\rk)\rightarrow \C,\qquad g\mapsto  \lambda\left(\left[ \begin{array}{cc}
              g&0\\ 0&1
              \end{array}
              \right].
                 v\right).\]
As in \cite{FJ}, consider the following integral:
\[
  \oZ_\lambda(v, s):=\int_{\GL_n(\rk)} \phi_{\lambda, v}(g)\cdot\abs{\det(g)}^{s-\frac{1}{2}}\,\od g, \quad s\in \C.
\]
Here and throughout this subsection, all the measures occurring are Haar measures.

\begin{lem}\label{stablem}
The set $\{\phi_{\lambda, v}\mid v\in \pi\}$ is stable under the multiplications by Schwartz functions on $\GL_n(\rk)$.
\end{lem}
\begin{proof}
Let $\phi$ be a Schwartz function on $\GL_n(\rk)$. Then there is a Schwartz funcntion $\hat \phi$ on $\oM_n(\rk)$ such that
 \[
 \phi(g)=  \int_{\oM_n(\rk)} \psi_\rk(\tr(gh)) \hat \phi(h) \od \!h, \quad \textrm{for all } g\in \GL_n(\rk),
 \]
 where $\psi_\rk$ is a non-trivial unitary character of $\rk$ as in \eqref{psik}.
For each $v\in \pi$, put
 \[
   v':=\int_{\oM_n(\rk)}  \hat \phi(h) \left[ \begin{array}{cc}
              1&h\\ 0&1
              \end{array}
              \right]
.v \od \!h\in \pi.
 \]
 Then it is easy to check that $\phi\cdot  \phi_{\lambda, v}=\phi_{\lambda, v'}$.
\end{proof}

Lemma \ref{stablem} has  the following obvious consequence.
\begin{lem}\label{stablem2}
Assume that $\lambda$ is non-zero.  Then for every $s\in \C$, there is a vector $v\in \pi$ such that the integral  $\oZ_\lambda(v, s)$ is absolutely convergent and non-zero.
\end{lem}

Let $\mathrm K_{2n}$ denote a fixed maximal compact subgroup of $\GL_{2n}(\rk)$.
For each  $h=[h_{i,j}]_{1\leq i,j\leq n}\in \GL_n(\rk)$, put
\[
||h||:=1+\sum_{1\leq i,j\leq n} \abs{h_{i,j}}+\abs{\det{h}}^{-1}.
\]
\begin{lem}\label{bounds}
Let $v_0\in \pi$. Then there exists a positive integer $N$ such that
\[
 \left|\lambda\left(
              \left[ \begin{array}{cc}
              h&0\\ 0&1
              \end{array}
              \right].( k  .v_0)\right)\right|\leq ||h||^N
\]
for all $h\in \GL_n(\rk)$ and all $k\in \mathrm K_{2n}$.
\end{lem}
\begin{proof}
When $\rk$ is archimedean, the lemma follows from the moderate growth condition on the Cassleman-Wallach representation $\pi$. When $\rk$ is non-archimedean, it suffices to show that
 there exists $N>0$ such that
\[
 \left|\lambda\left(
              \left[ \begin{array}{cc}
              h&0\\ 0&1
              \end{array}
              \right].v_0\right)\right|\leq ||h||^N
\]
for all $h\in \GL_n(\rk)$. A stronger form  of this result is proved in \cite[Lemma 6.1]{JR}, with
 the assumption that  $\psi_{\oS_n}$ has trivial restriction to $\oD_n(\rk)$ (see \eqref{dn}). But their proof works without this assumption.
\end{proof}

\begin{lem}\label{converg0}
When the real part of $s$ is sufficiently large, the integral $\oZ_\lambda(v,s)$ is absolutely convergent for all $v\in \pi$, and the resulting linear functional  $v\mapsto \oZ_\lambda(v,s)$  on $\pi$ is continuous in the archimedean case.
\end{lem}
\begin{proof}

Fix a non-zero element $v_0\in \pi$. Then
\[
 \CS(\GL_{2n}(\rk))\rightarrow \pi, \quad   \phi\mapsto \phi.v_0:=\int_{\GL_{2n}(\rk)} \phi(g)\, g.v_0 \od\! g
\]
is a surjective linear map, and it is open and continuous in the archimedean case. Here  ``$\CS$" stands for the space of Schwartz functions as usual. Define a linear map (which is continuous in the archimedean case)
\[
   \CS(\GL_{2n}(\rk))\rightarrow  \CS(\oM_n(\rk)\times \GL_{n}(\rk)\times \GL_n(\rk)\times \mathrm K_{2n}), \quad \phi\mapsto \hat \phi
\]
by
\[
  \hat \phi(y,a,b,k):=\int_{\oM_n(\rk)} \psi_\rk(\tr (yx))   \phi\left(\left[
  \begin{array}{cc}
              1&x\\ 0&1
              \end{array}
              \right]\cdot
             \left[ \begin{array}{cc}
              a&0\\ 0&b
              \end{array}
              \right] \cdot k
              \right) \od\! x.
\]
Then
\begin{eqnarray*}
  && \lambda\left(  \left[ \begin{array}{cc}
              h&0\\ 0&1
              \end{array}
              \right] .(\phi.v_0)\right)\\
  &=&\int_{\GL_n(\rk)} \int_{\GL_n(\rk)} \int_{\mathrm K_{2n}} \int_{\oM_n(\rk)}
    \abs{\det a}^{-n} \cdot \abs{\det b}^n\cdot
  \phi\left(\left[
  \begin{array}{cc}
              1&x\\ 0&1
              \end{array}
              \right]\cdot
             \left[ \begin{array}{cc}
              a&0\\ 0&b
              \end{array}
              \right] \cdot k
              \right)\\
              &&\cdot  \lambda\left( \left( \left[ \begin{array}{cc}
              h&0\\ 0&1
              \end{array}
              \right] \cdot
            \left[ \begin{array}{cc}
              1&x\\ 0&1
              \end{array}
              \right]\cdot
             \left[ \begin{array}{cc}
              a&0\\ 0&b
              \end{array}
              \right] \cdot k\right) .v_0\right) \od\! x \od\! k \od\! a \od\! b \\
             & =&\int_{\GL_n(\rk)} \int_{\GL_n(\rk)} \int_{\mathrm K_{2n}}
   \hat \phi\left(h,
             a,b, k
              \right)\\
              &&\cdot
              \psi_{\oS_n}\left( \left[ \begin{array}{cc}
              b&0\\ 0&b
              \end{array}
              \right] \right)\cdot
                \lambda\left( \left( \left[ \begin{array}{cc}
              b^{-1}ha&0\\ 0&1
              \end{array}
              \right] \cdot k\right) .v_0\right) \od\! k \od\! a \od\! b.
\end{eqnarray*}
By Lemma \ref{bounds}, the absolute values of this integral is bounded by
\[
  \int_{\GL_n(\rk)} \int_{\GL_n(\rk)} \int_{\mathrm K_{2n}}
 \abs{ \hat \phi\left(h, a,b, k
              \right)}\\
             \cdot
              ||h||^N\cdot ||a||^N\cdot ||b||^N \od\! k \od\! a \od\! b,
\]
where $N$ is a positive integer which is independent of $\phi$. Now the lemma follows easily, as in the proof of the convergence of Godement-Jacquet zeta integrals.
 \end{proof}

 Now we are ready to prove  Theorem \ref{unique0}. Let  $\CL$ be a finite dimensional subspace of $\Hom_{\oS_n(\rk)}(\pi, \psi_{\oS_n})$. By Lemma \ref{converg0} and Lemma \ref{stablem2},  for all $s\in \C$ whose real part is sufficiently large, we have a well defined injective linear map
 \begin{eqnarray}\label{heh}
  \CL\rightarrow \Hom_{\GL_n(\rk)\times \GL_n(\rk)}(\pi, \chi_s),\quad
\lambda\mapsto \oZ_\lambda(\cdot,s),
\end{eqnarray}
where
  $\chi_s$ is the character of $\GL_n(\rk)\times \GL_n(\rk)$ defined by
  \[\chi_s\left(\left[\begin{array}{cc}
              a&0\\ 0&b
              \end{array}
              \right]\right)=\psi_{\oS_n}\left(\left[\begin{array}{cc}
              b&0\\ 0&b
              \end{array}
              \right]\right)\cdot \abs{\det(ba^{-1})}^{s-\frac{1}{2}},\quad a,b\in \GL_n(\rk).\]
Then Theorem \ref{uniquel} implies that the space $\CL$ is at most one dimensional. This proves Theorem \ref{unique0}.

\end{document}